\documentclass[a4paper,12pt]{article}

\usepackage[margin=1.2in]{geometry}

\usepackage{amsmath}
\usepackage{algorithmic}
\usepackage{algorithm}
\usepackage{amsthm}
\usepackage{amsfonts}
\usepackage{comment}
\usepackage{graphicx}
\usepackage{hyperref}
\hypersetup{hidelinks}
\usepackage{color}
\usepackage{subcaption}
\usepackage{array,multirow}

\newtheorem{assumption} {Assumption}
\newtheorem{theorem} {Theorem}
\newtheorem{lemma} {Lemma}

\newtheorem{remark} {Remark}

\def\x{{\mathbf{x}}}

\def\e{{\mathbf{e}}}

\def\a{{\mathbf{a}}}
\def\u{{\mathbf{u}}}
\def\v{{\mathbf{v}}}
\def\z{{\mathbf{z}}}
\def\w{{\mathbf{w}}}

\def\y{{\mathbf{y}}}
\def\q{{\mathbf{q}}}
\def\p{{\mathbf{p}}}
\def\b{{\mathbf{b}}}

\def\X{{\mathbf{X}}}
\def\Y{{\mathbf{Y}}}
\def\A{{\mathbf{A}}}
\def\M{{\mathbf{M}}}
\def\I{{\mathbf{I}}}
\def\B{{\mathbf{B}}}
\def\C{{\mathbf{C}}}
\def\V{{\mathbf{V}}}
\def\Z{{\mathbf{Z}}}
\def\W{{\mathbf{W}}}
\def\U{{\mathbf{U}}}
\def\Q{{\mathbf{Q}}}
\def\P{{\mathbf{P}}}

\def\D{{\mathbf{D}}}
\def\bR{{\mathbf{R}}}

\def\matE{{\mathbf{E}}}

\newcommand{\mX}{\mathcal{X}}

\newcommand{\mS}{\mathcal{S}}

\newcommand{\mbS}{\mathbb{S}}

\newcommand{\trace}{\textrm{Tr}}
\newcommand{\rank}{\textrm{rank}}
\newcommand{\reals}{\mathbb{R}}

\DeclareMathOperator*{\argmin}{argmin}
\DeclareMathOperator*{\argmax}{argmax}

\def\H{{\mathbf{H}}}
\def\bLambda{{\mathbf{\Lambda}}}

\title{Linear Convergence of a Frank-Wolfe-type Method over the Spectrahedron without Strict Complementarity}
\date{\vspace{-5pt} Faculty of Data and Decision Sciences  \vspace{3pt}\\ Technion - Israel Institute of Technology}
\author{Dan Garber \\ {\small{dangar@technion.ac.il}} }

\begin{document}
\maketitle

\begin{abstract}
We consider smooth convex minimization over the spectrahedron using Frank-Wolfe-type methods based only on extreme-eigenvector computations. In our recent work \cite{garber2026randomized} we presented the first ambient-dimension-independent linear convergence rate under quadratic growth. However, the method makes an additional strong strict complementarity assumption, it is randomized, its linear rate holds only after a burn-in phase and in expectation, and it requires the objective smoothness constant. We show that these limitations can be removed. Assuming quadratic growth and that all optimal solutions have the same rank, but without assuming strict complementarity, we give a deterministic and parameter-free Frank-Wolfe-type method with a global ambient-dimension-independent linear convergence rate. 
\end{abstract}

\section{Introduction}\label{sec:introduction}

We consider the problem
\begin{align}\label{eq:problem}
\min_{\X\in\mS^n} f(\X),
\qquad
\mS^n := \{\X\in\mbS^n~|~\X\succeq 0,~\trace(\X)=1\}.
\end{align}
Throughout, $\mbS^n$ denotes the space of real $n\times n$ symmetric matrices and $\mbS_+^n$ its positive-semidefinite cone. The set $\mS^n$ is known as the spectrahedron in $\mbS^n$. 
For $\A\in\mbS^n$, we order the eigenvalues as $\lambda_1(\A)\geq\cdots\geq\lambda_n(\A)$. We use $\langle{\A,\B}\rangle:=\trace(\A\B)$ for the trace inner product, and $\A^\dagger$ denotes the Moore--Penrose pseudoinverse of $\A$. Throughout, $\Vert{\cdot}\Vert$ denotes an arbitrary unitarily invariant matrix norm and $\Vert{\cdot}\Vert_*$ its corresponding dual norm with respect to the trace inner product. We use $\Vert{\cdot}\Vert_2$ for the Euclidean norm of column vectors and the spectral norm of matrices, and $\Vert{\cdot}\Vert_F$ for the Frobenius norm of matrices. We assume that $f$ is convex and $\beta$-smooth over $\mS^n$ with respect to this primal-dual norm pair, i.e.,
\begin{align}\label{eq:smoothness}
\Vert{\nabla f(\X)-\nabla f(\Y)}\Vert_*\leq \beta\Vert{\X-\Y}\Vert
\qquad
\forall\X,\Y\in\mS^n.
\end{align}
In particular, the following standard inequality holds:
\begin{align}\label{eq:smooth-upper}
f(\Y)\leq f(\X)+\langle{\Y-\X,\nabla f(\X)}\rangle+\frac{\beta}{2}\Vert{\Y-\X}\Vert^2.
\end{align}
Let $\mX^*:=\argmin_{\X\in\mS^n}f(\X)$ denote the optimal set and let $f^*$ denote the optimal value. 
For our linear convergence result we make the following two assumptions.

\begin{assumption}[quadratic growth]\label{ass:qg}
There exists a scalar $\alpha>0$ such that for any $\X\in\mS^n$,
\begin{align}\label{eq:qg}
\min_{\X^*\in\mX^*}\Vert{\X-\X^*}\Vert^2
\leq
\frac{2}{\alpha}\left({f(\X)-f^*}\right).
\end{align}
\end{assumption}

\begin{assumption}[common rank of optimal solutions]\label{ass:optimal-rank}
There exists an integer $r^*\in\{1,\ldots,n\}$ such that $\rank(\X^*)=r^*$ for all $\X^*\in\mX^*$.
 Accordingly we define $\lambda_{r^*}^* := \min_{\X^*\in\mX^*}\lambda_{r^*}(\X^*)$.
\end{assumption}
For more discussion of Assumption \ref{ass:optimal-rank} and its implications see Lemma \ref{lem:common-face} in the sequel and the text following it.

Define also the quantities:
\begin{eqnarray*}
&\widetilde G :=\sup_{\X\in\mS^n}\left\Vert{\nabla f(\X)-\lambda_n(\nabla f(\X))\I}\right\Vert_2,&\\
&{D}:=\sup_{\X,\Y\in\mS^n}\Vert{\X-\Y}\Vert, \qquad \kappa_{r^*} :=
\sup_{\substack{\A\in\mbS^n,~\A\neq\mathbf{0}\\ \rank(\A)\leq2r^*}}
\frac{\Vert{\A}\Vert_F}{\Vert{\A}\Vert}.&
\end{eqnarray*}

Problem \eqref{eq:problem} has numerous applications in various fields including statistics, machine learning, discrete optimization, and more. It is in particular interesting in the context of the Frank-Wolfe optimization method \cite{frank1956algorithm, levitin1966constrained, jaggi2013revisiting}, since the Frank-Wolfe linear optimization step over $\mS^n$ amounts to computing a unit eigenvector corresponding to the smallest eigenvalue of the current gradient \cite{jaggi2010simple, hazan2008sparse}. This can be considerably cheaper than the projection step used by standard projection-based first-order methods as projection onto $\mS^n$ requires, in general, a full eigendecomposition of an $n\times n$ symmetric matrix, which in practical implementations takes $O(n^3)$ time. Moreover, Frank-Wolfe applies rank-one updates, which often allow efficient low-rank implementations. Its main drawback however, is its worst-case convergence rate. Standard Frank-Wolfe satisfies $f(\X_t)-f^* = O\left({\beta{D}^2/t}\right)$, where $t$ is the iteration counter,
and this dependence on $1/t$ is worst-case tight even under quadratic growth (see for instance a typical lower bound in \cite{jaggi2011convex}). Thus, faster rates require exploiting additional structure, and a dependence on the rank $r^*$ of optimal solutions is natural in this setting. Indeed Problem \eqref{eq:problem} has received significant interest in recent years, in particular in the context of Frank-Wolfe-type methods, see for instance \cite{jaggi2010simple, hazan2008sparse, NIPS2016_df877f38, freund2017extended, allen2017linear, garber2023linear, ding2020spectral, danon2022frank}. We refer the interested reader to our very recent work \cite{garber2026randomized} for a more in-depth discussion of related works.  

Our recent work \cite{garber2026randomized} proposed a Frank-Wolfe-type method for this problem that combines the standard Frank-Wolfe step with an away step specialized to the spectrahedron and also a specialized randomized pairwise step. The method requires only three leading-eigenvector computations per iteration and, under quadratic growth and strict complementarity, converges linearly in expectation after a finite burn-in phase, with both the burn-in and the linear rate independent of the ambient dimension. However, the method is randomized, its linear rate holds only in expectation and only after the burn-in phase, it requires knowledge of the smoothness parameter $\beta$, and its analysis relies on strict complementarity, which is a non-trivial assumption. 

In this paper we remove all of these limitations. We present a deterministic and parameter-free Frank-Wolfe-type method which requires only three extreme-eigenvector computations per iteration and simple line-searches. Without Assumptions \ref{ass:qg} and \ref{ass:optimal-rank}, it retains the standard $O(\beta{D}^2/t)$ Frank-Wolfe rate. Under these assumptions it converges globally at a linear rate, without a burn-in phase and without strict complementarity. In particular, for some universal numerical constant $c>0$,
\begin{align*}
f(\X_t)-f^*
\leq
\left(f(\X_1)-f^*\right)
\exp\left(
-c(t-1)
\min\left\{
1,
\frac{\alpha\lambda_{r^*}^*}
{r^*\kappa_{r^*}^2\left(\beta{D}^2+\widetilde G\right)}
\right\}
\right).
\end{align*}
Thus, up to universal constants, the contraction rate has no explicit dependence on the ambient dimension $n$. In particular, in case $\Vert{\cdot}\Vert$ is either the Frobenius norm or the Nuclear norm (sum of singular values), we simply have $\kappa_{r^*} =1$.

While our method builds on our recent work \cite{garber2026randomized}, a key ingredient is a new deterministic pairwise step which rotates a rank-one component supported by the current iterate in a descent direction, which is quite different than the randomized pairwise step introduced in \cite{garber2026randomized}. This yields first-order progress in situations in which both the Frank-Wolfe and away steps may provide only  insufficient progress. Together with a new geometric argument based on quadratic growth rather than a gradient eigengap, this is what allows us to obtain the global deterministic linear rate.

\subsection{Consequences of Assumption \ref{ass:optimal-rank}}\label{sec:setup}

The following simple observation records two consequences of Assumption \ref{ass:optimal-rank}.

\begin{lemma}[common optimal face]\label{lem:common-face}
Under Assumption \ref{ass:optimal-rank}, it holds that $\lambda_{r^*}^*>0$. Moreover, there exists a matrix $\V^*\in\reals^{n\times r^*}$ with $\V^{*\top}\V^*=\I$ such that
\begin{align}\label{eq:common-face}
\mX^*
\subset
\left\{
\V^*\mathbf{S}\V^{*\top}
~\middle|~
\mathbf{S}\in\mbS^{r^*},~\mathbf{S}\succ0,~\trace(\mathbf{S})=1
\right\}.
\end{align}
\end{lemma}

\begin{proof}
Fix $\X_1^*,\X_2^*\in\mX^*$. By convexity of $f$, the midpoint $(\X_1^*+\X_2^*)/2$ is also optimal and therefore has rank $r^*$. Since both matrices are positive semidefinite,
\begin{align*}
\textrm{Ker}(\X_1^*+\X_2^*)
=
\textrm{Ker}(\X_1^*)\cap\textrm{Ker}(\X_2^*),
\end{align*}
and hence
\begin{align*}
\textrm{Im}(\X_1^*+\X_2^*)
=
\textrm{Im}(\X_1^*)+\textrm{Im}(\X_2^*).
\end{align*}
The two images have dimension $r^*$, while the image of their sum also has dimension $r^*$. It follows that
$\textrm{Im}(\X_1^*)=\textrm{Im}(\X_2^*)$.
Thus, all optimal solutions have a common image. Choosing an orthonormal basis $\V^*$ for this image gives \eqref{eq:common-face}; the matrix $\mathbf{S}$ is positive definite because every optimal solution has rank $r^*$.

Since $\lambda_{r^*}(\X)$ is strictly positive at every point of $\mX^*$, its minimum over $\mX^*$ is strictly positive and so, $\lambda_{r^*}^*>0$.
\end{proof}

Assumption \ref{ass:optimal-rank} is in fact equivalent to requiring that the nonzero eigenvalues of all optimal solutions are uniformly bounded away from zero. Indeed, Lemma \ref{lem:common-face} proves one direction. Conversely, suppose there exists $\mu>0$ such that every nonzero eigenvalue of every $\X^*\in\mX^*$ is at least $\mu$. If two optimal solutions had different images, then for some endpoint $\X_1^*$ the matrices $(1-\tau)\X_1^*+\tau\X_2^*,~ \tau>0$,
would have rank strictly larger than $\X_1^*$. Since these matrices are optimal by convexity and converge to $\X_1^*$ as $\tau\downarrow0$, their additional positive eigenvalues must converge to zero, contradicting the uniform lower bound $\mu$. Hence all optimal solutions have the same image, and therefore the same rank.

\subsection{Organization of this paper}
In Section \ref{sec:algorithm} we present our algorithm. In Section \ref{sec:implementation} we detail a concrete efficient implementation in terms of memory and runtime. In Section \ref{subsec:pairwise-need} we demonstrate why the inclusion of a pairwise step in the algorithm is needed. We show that relying only on the standard Frank-Wolfe step or the away step (which are sufficient for linearly converging Frank-Wolfe methods over polytopes \cite{garber2016linearly, lacoste2015global}) is not enough to guarantee sufficient objective decrease. In Section \ref{sec:analysis} we state our main convergence rate theorem and prove it.

\section{The Algorithm}\label{sec:algorithm}
Our algorithm is given below as Algorithm \ref{alg:newFW}. Fix an iteration $t$ and write
\begin{align*}
\nabla_t:=\nabla f(\X_t),
\qquad
\widetilde{\nabla}_t:=\nabla_t-\lambda_n(\nabla_t)\I\succeq0.
\end{align*}
Since all matrices in $\mS^n$ have unit trace, for any $\X,\Y\in\mS^n$,
$\langle{\widetilde{\nabla}_t,\Y-\X}\rangle =\langle{\nabla_t,\Y-\X}\rangle$.
Thus, the shift by $\lambda_n(\nabla_t)\I$ does not change any directional derivative between feasible points. We let $\v_{t,+}$ be a unit eigenvector corresponding to $\lambda_n(\nabla_t)$. Additionally, we compute
\begin{align}\label{eq:vminus-def}
\v_{t,-}
\in
\argmax_{\v\in\textrm{Im}(\X_t):\Vert{\v}\Vert_2=1}
\v^{\top}\widetilde{\nabla}_t\v,
\end{align}
as well as
\begin{align}\label{eq:u-ratio-def}
\u_t
\in
\argmax_{\u\in\textrm{Im}(\X_t):\Vert{\u}\Vert_2=1}
\frac{\u^{\top}\widetilde{\nabla}_t^2\u}{\u^{\top}\X_t^{\dagger}\u}.
\end{align}
We define the following three quantities (certificates) used by the algorithm:
\begin{align}\label{eq:algorithm-certificates}
a_t&:=\langle{\widetilde{\nabla}_t,\X_t}\rangle^2,\nonumber\\
b_t&:=\left({\v_{t,-}^{\top}\widetilde{\nabla}_t\v_{t,-}-\langle{\widetilde{\nabla}_t,\X_t}\rangle}\right)^2,\nonumber\\
c_t&:=
\frac{\u_t^{\top}\widetilde{\nabla}_t^2\u_t-(\u_t^{\top}\widetilde{\nabla}_t\u_t)^2}
{\u_t^{\top}\X_t^{\dagger}\u_t}.
\end{align}
These will correspond to the possible decrement in function value on iteration $t$.


The algorithm uses the vectors $\v_{t,+}, \v_{t,-}, \u_t$ to construct three possible update steps as we now detail.
\medskip

\noindent\textbf{Frank-Wolfe step:}
As in our previous work \cite{garber2026randomized}, the Frank-Wolfe step moves from $\X_t$ towards the extreme point $\v_{t,+}\v_{t,+}^{\top}$, exactly as in the standard Frank-Wolfe method over the spectrahedron. By definition, the quantity $\langle{\widetilde{\nabla}_t,\X_t}\rangle = \langle{\nabla_t, \X_t - \v_{t,+}\v_{t,+}^{\top}}\rangle$ is precisely the Frank-Wolfe gap at $\X_t$.

\medskip
\noindent\textbf{Away and Drop steps:}
We use the same away steps introduced in our previous work \cite{garber2026randomized}.
The away step decreases the weight of the rank-one matrix $\v_{t,-}\v_{t,-}^{\top}$, where $\v_{t,-}$ is chosen according to \eqref{eq:vminus-def}. For the nondegenerate case $\rank(\X_t) > 1$, we set
\begin{align}\label{eq:away-s}
s_t:=\left({\v_{t,-}^{\top}\X_t^{\dagger}\v_{t,-}-1}\right)^{-1}.
\end{align}
This is the maximal scaling for the parametrization used below. Indeed, for $\theta\in[0,1]$, setting
$\eta:=\frac{\theta s_t}{1+\theta s_t}$
gives
\begin{align*}
\X_t+\theta s_t(\X_t-\v_{t,-}\v_{t,-}^{\top})
=
\frac{1}{1-\eta}
\left({\X_t-\eta\v_{t,-}\v_{t,-}^{\top}}\right),
\end{align*}
When $\theta=1$, we have $\eta = (\v_{t,-}^{\top}\X_t^{\dagger}\v_{t,-})^{-1}$, and the rank is reduced by one (see Lemma \ref{lem:stepsize} below). We refer to such a step as a \emph{drop step}.

\medskip
\noindent\textbf{Pairwise step:}
As in our previous work \cite{garber2026randomized} we also consider a specialized pairwise step, however one that is very different from the (randomized) step suggested in \cite{garber2026randomized}.
The pairwise step first removes the maximal feasible weight in the direction $\u_t\u_t^{\top}$ and then places the same weight on a nearby rank-one matrix. In the nondegenerate case $(\I-\u_t\u_t^{\top})\widetilde\nabla_t\u_t \neq \mathbf{0}$ set
\begin{align}\label{eq:pairwise-gamma-p}
\gamma_t:=\left({\u_t^{\top}\X_t^{\dagger}\u_t}\right)^{-1},
\quad
\p_t:=
\frac{(\I-\u_t\u_t^{\top})\widetilde\nabla_t\u_t}
{\Vert{(\I-\u_t\u_t^{\top})\widetilde\nabla_t\u_t}\Vert_2}, \quad
\v_t(\theta):=\frac{\u_t-\theta\p_t}{\sqrt{1+\theta^2}}.
\end{align}
Note that $\p_t\perp\u_t$ and therefore $\Vert{\v_t(\theta)}\Vert_2=1$, i.e., $\v_t(\theta)$ is a rotation of $\u_t$.

While the definitions of $\u_t, \p_t, \v_t(\theta)$ may seem unintuitive, we postpone the justification to later on, and as part of the convergence analysis. In Lemma \ref{lem:pairwise-decrease} we give a step-by-step derivation that justifies these choices.

\medskip
\noindent\textbf{Line-search and choice of step:}
We perform step-size search for each of the three steps. For each step we define a step-size function
\begin{align}\label{eq:trial-path}
\begin{aligned}
\Phi_t^{\rm FW}(\theta)
&:=\X_t+\theta(\v_{t,+}\v_{t,+}^{\top}-\X_t),\\
\Phi_t^{\rm A}(\theta)
&:=\X_t+\theta s_t(\X_t-\v_{t,-}\v_{t,-}^{\top}),\\
\Phi_t^{\rm P}(\theta)
&:=\X_t+\gamma_t\left({\v_t(\theta)\v_t(\theta)^{\top}-\u_t\u_t^{\top}}\right),
\qquad \theta\in[0,1].
\end{aligned}
\end{align}
Their corresponding Armijo coefficients are
\begin{align}\label{eq:armijo-coefficient}
\begin{aligned}
\delta_t^{\rm FW}
&:=\frac{1}{2}\langle{\widetilde{\nabla}_t,\X_t}\rangle,\\
\delta_t^{\rm A}
&:=\frac{1}{2}s_t\left({\v_{t,-}^{\top}\widetilde{\nabla}_t\v_{t,-}-\langle{\widetilde{\nabla}_t,\X_t}\rangle}\right),\\
\delta_t^{\rm P}
&:=\gamma_t\Vert{(\I-\u_t\u_t^{\top})\widetilde\nabla_t\u_t}\Vert_2.
\end{aligned}
\end{align}
Let $\Theta:=\{0\}\cup\{2^{-j}:j=0,1,\ldots\}$. For each nondegenerate case $i\in\{{\rm FW},{\rm A},{\rm P}\}$, the Armijo search returns the largest $\theta\in\Theta$ satisfying
\begin{align}\label{eq:armijoCond}
f(\Phi_t^i(\theta))\leq f(\X_t)-\theta\delta_t^i.
\end{align}

We use the following conventions in the two degenerate cases. If $\rank(\X_t) = 1$ an away-step is not possible. Also, if $\delta_t^{\rm A} =0$, an away-step is not sensible. In these cases we simply set $\Phi_t^{\rm A}(\theta)\equiv\X_t$, $\delta_t^{\rm A}=0$, and the away search returns $\theta_t^{\rm A}=0$ and $\X_t^{\rm A}=\X_t$. If $(\I-\u_t\u_t^{\top})\widetilde\nabla_t\u_t=\mathbf{0}$, then the vector $\p_t$ in the pairwise step is undefined and we set $\Phi_t^{\rm P}(\theta)\equiv\X_t$ and $\delta_t^{\rm P}=0$. 

If $\max\{a_t,b_t,c_t\}=0$, the method terminates. Otherwise, the three Armijo searches are performed. If the away search accepts the full step, the algorithm takes this step and the iteration is called a \emph{drop iteration}. Otherwise, the algorithm takes the point with the smallest objective value among the three Armijo points. Such an iteration is called a \emph{non-drop iteration}.

\begin{remark}
In principle, all three Armijo step-size searches could be readily replaced with exact line-search w.r.t. the objective function $f$. Indeed for the Frank-Wolfe and away steps this is completely standard in the literature. However, for the pairwise step the dependence of $\Phi_t^{\rm P}(\theta)$ on the parameter $\theta$ makes such a procedure more complex. 
\end{remark}

\begin{algorithm}[H]
\begin{algorithmic}
\caption{Deterministic and Parameter-free Spectrahedron Frank-Wolfe with Away and Pairwise steps}\label{alg:newFW}
\STATE $\X_1\gets\x_1\x_1^\top$ for an arbitrary unit vector $\x_1\in\reals^n$
\FOR{$t=1,2,\ldots$}
\STATE $\v_{t,+}\gets$ leading eigenvector of $-\nabla{}f(\X_t)$
\STATE $\widetilde{\nabla}_t\gets\nabla{}f(\X_t)-(\v_{t,+}^\top\nabla{}f(\X_t)\v_{t,+})\I$
\STATE $\displaystyle \v_{t,-}\gets\argmax_{\v\in\textrm{Im}(\X_t):\Vert{\v}\Vert_2=1}\v^\top\widetilde{\nabla}_t\v$
\STATE $\displaystyle \u_t\gets\argmax_{\u\in\textrm{Im}(\X_t):\Vert{\u}\Vert_2=1}\frac{\u^\top\widetilde{\nabla}_t^2\u}{\u^\top\X_t^\dagger\u}$
\STATE $\displaystyle a_t\gets\langle{\widetilde{\nabla}_t,\X_t}\rangle^2$, $\displaystyle b_t\gets\left({\v_{t,-}^{\top}\widetilde{\nabla}_t\v_{t,-}-\langle{\widetilde{\nabla}_t,\X_t}\rangle}\right)^2$, $\displaystyle c_t\gets\frac{\u_t^{\top}\widetilde{\nabla}_t^2\u_t-(\u_t^{\top}\widetilde{\nabla}_t\u_t)^2}{\u_t^{\top}\X_t^{\dagger}\u_t}$

\IF{$\max\{a_t,b_t,c_t\}=0$}
\RETURN
\ENDIF
\STATE form $\Phi_t^{\rm FW},\Phi_t^{\rm A},\Phi_t^{\rm P}$ and $\delta_t^{\rm FW},\delta_t^{\rm A},\delta_t^{\rm P}$ according to Eq. \eqref{eq:trial-path}--\eqref{eq:armijo-coefficient}
\STATE compute $\theta_t^{\rm FW},\theta_t^{\rm A},\theta_t^{\rm P}$ and $\X_t^{\rm FW},\X_t^{\rm A},\X_t^{\rm P}$ by the three Armijo searches
\IF{$\theta_t^{\rm A}=1$}
\STATE $\X_{t+1}\gets\X_t^{\rm A}$ \COMMENT{drop step}
\ELSE
\STATE $\displaystyle \X_{t+1}\gets\argmin_{\X\in\{\X_t^{\rm FW},\X_t^{\rm A},\X_t^{\rm P}\}} f(\X)$
\ENDIF
\ENDFOR
\end{algorithmic}
\end{algorithm}



As in \cite{garber2026randomized}, the feasibility of trial points (and in particular the produced iterates) of Algorithm \ref{alg:newFW} follows directly from the following lemma (a proof is given in the appendix for completeness).
\begin{lemma}\label{lem:stepsize}
Let $\X\in\mbS^n_+$ and $\v\in\textrm{Im}(\X)$, $\v\neq\mathbf{0}$. Consider the matrix $\Y=\X-\lambda\v\v^{\top}$ for some $\lambda\geq0$. If $\lambda\leq(\v^{\top}\X^{\dagger}\v)^{-1}$, then $\Y\succeq0$. Moreover, if $\lambda=(\v^{\top}\X^{\dagger}\v)^{-1}$, then $\rank(\Y)=\rank(\X)-1$ and $\X^{\dagger}\v\in\textrm{Ker}(\Y)$.
\end{lemma}

\begin{lemma}[feasibility of Algorithm \ref{alg:newFW}]\label{lem:feasibility}
Every trial point considered by Algorithm \ref{alg:newFW}, and hence every iterate produced by it, belongs to $\mS^n$.
\end{lemma}



\subsection{Implementation details}\label{sec:implementation}
There are many possible efficient implementations of Algorithm \ref{alg:newFW} using various efficient linear algebra tools in order to store and update the state of the algorithm (two such approaches are already discussed in our previous work  \cite{garber2026randomized}). Here we present one simple approach that uses  $O(n\rank(\X_t))$ memory, instead of the worst-case $O(n^2)$, to store the state of the algorithm on iteration $t$. 

Denote $r_t:=\rank(\X_t)$ and maintain a thin factorization in the form
\begin{align*}
\X_t=\U_t\mathbf{S}_t\U_t^{\top},
\quad
\U_t\in\reals^{n\times r_t},
\quad
\U_t^{\top}\U_t=\I,
\quad
\mathbf{S}_t\succ0,
\quad
\W_t:=\mathbf{S}_t^{-1} \succ 0.
\end{align*}
The pseudoinverse and the projection matrix onto $\textrm{Im}(\X_t)$ are then represented implicitly as $\X_t^{\dagger}=\U_t\W_t\U_t^{\top}$ and $\Pi_t=\U_t\U_t^{\top}$. Thus, vector products with $\X_t$, $\X_t^{\dagger}$, and $\Pi_t$ require $O(nr_t+r_t^2)$ arithmetic operations, and the maintained state requires $O(nr_t+r_t^2)=O(nr_t)$ memory.

\medskip
\noindent\textbf{Computing the vectors $\v_{t,+},\v_{t,-},\u_t$:}
The computation of $\v_{t,+}$ is exactly the standard Frank-Wolfe leading-eigenvector computation w.r.t. $-\nabla_t$. As in \cite{garber2026randomized}, the constrained computation of $\v_{t,-}$ in \eqref{eq:vminus-def} can also be reduced to a standard leading-eigenvector computation. Indeed, consider the symmetric matrix $\Pi_t\widetilde{\nabla}_t\Pi_t+\Pi_t$. It vanishes on $\textrm{Im}(\X_t)^{\perp}$, while on any unit vector $\v\in\textrm{Im}(\X_t)$ its Rayleigh quotient equals $1+\v^{\top}\widetilde{\nabla}_t\v\geq1$. Hence every leading eigenvector belongs to $\textrm{Im}(\X_t)$ and solves \eqref{eq:vminus-def}.

The computation of $\u_t$ in \eqref{eq:u-ratio-def} can similarly be reduced to a standard leading-eigenvector computation. For every nonzero $\u\in\textrm{Im}(\X_t)$ write $\u=\X_t^{1/2}\y$. Then
\begin{align*}
\frac{\u^{\top}\widetilde{\nabla}_t^2\u}{\u^{\top}\X_t^{\dagger}\u}
=
\frac{\y^{\top}\X_t^{1/2}\widetilde{\nabla}_t^2\X_t^{1/2}\y}{\Vert{\y}\Vert_2^2}.
\end{align*}
Therefore, the optimal value in \eqref{eq:u-ratio-def} is $\lambda_1(\X_t^{1/2}\widetilde{\nabla}_t^2\X_t^{1/2})=\lambda_1(\widetilde{\nabla}_t\X_t\widetilde{\nabla}_t)$, where the equality follows since the two matrices have the same nonzero eigenvalues. If $\z_t$ is a leading unit eigenvector of $\widetilde{\nabla}_t\X_t\widetilde{\nabla}_t$ corresponding to a positive eigenvalue, then
\begin{align}\label{eq:u-from-leading}
\u_t
=
\frac{\X_t\widetilde{\nabla}_t\z_t}{\Vert{\X_t\widetilde{\nabla}_t\z_t}\Vert_2}
\end{align}
solves \eqref{eq:u-ratio-def}. If $\lambda_1(\widetilde{\nabla}_t\X_t\widetilde{\nabla}_t)=0$, then the numerator in \eqref{eq:u-ratio-def} vanishes for every $\u\in\textrm{Im}(\X_t)$, and any unit vector in this image is optimal. A matrix-vector product with $\widetilde{\nabla}_t\X_t\widetilde{\nabla}_t$ requires two applications of $\widetilde{\nabla}_t$ and $O(nr_t+r_t^2)$ additional arithmetic operations.

Once $\v_{t,+}$, $\v_{t,-}$, and $\u_t$ have been computed, the quantities
$s_t$, $\gamma_t$, $\p_t$, and $c_t$ can be obtained directly from the
maintained representation and a small number of additional matrix-vector
products.
In particular, for any $\v\in\textrm{Im}(\X_t)$, $\v^{\top}\X_t^{\dagger}\v=(\U_t^{\top}\v)^{\top}\W_t(\U_t^{\top}\v)$. Thus, $s_t$ and $\gamma_t$ require $O(nr_t+r_t^2)$ arithmetic operations. Moreover, after computing $\nabla_t\u_t$, the vector $\p_t$ and the quantity $c_t$ are obtained from $c_t=\gamma_t\Vert{(\I-\u_t\u_t^{\top})\nabla_t\u_t}\Vert_2^2$. The quantity $\v_{t,-}^{\top}\widetilde{\nabla}_t\v_{t,-}$ is available from the Rayleigh quotient computed with $\v_{t,-}$.

\medskip
\noindent\textbf{Updating the maintained representation:}
$\U_t$ is changed only when a direction enters or leaves $\textrm{Im}(\X_t)$, while $\mathbf{S}_t$ and $\W_t=\mathbf{S}_t^{-1}$ are updated on this image. Whenever the updated small core remains invertible, its inverse is obtained from $\W_t$ using the standard Sherman--Morrison or block-inverse formulas, and hence in $O(r_t^2)$ arithmetic operations. We therefore only describe explicitly the changes of the image and the rank-decreasing cases.

Suppose first that a Frank-Wolfe step with step-size $\theta:=\theta_t^{\rm FW}$ is accepted. Set $\a:=\U_t^{\top}\v_{t,+}$ and $\z:=\v_{t,+}-\U_t\a$. If $0<\theta<1$ and $\z=0$, the image is unchanged: $\U_{t+1}=\U_t$ and $\mathbf{S}_{t+1}=(1-\theta)\mathbf{S}_t+\theta\a\a^{\top}$. The matrix $\mathbf{S}_{t+1}$ is positive definite, and $\W_{t+1}=\mathbf{S}_{t+1}^{-1}$ is obtained from $\W_t$ by a scalar rescaling and the Sherman--Morrison formula in $O(r_t^2)$ time.

If $0<\theta<1$ and $\z\neq0$, one new direction enters the image. In this case set
\begin{align*}
\U_{t+1}
=
\left[\U_t,\frac{\z}{\Vert{\z}\Vert_2}\right], \qquad \mathbf{S}_{t+1}
=
\begin{pmatrix}
(1-\theta)\mathbf{S}_t+\theta\a\a^{\top}
&
\theta\Vert{\z}\Vert_2\a
\\
\theta\Vert{\z}\Vert_2\a^{\top}
&
\theta\Vert{\z}\Vert_2^2
\end{pmatrix}.
\end{align*}
This matrix is positive definite, and its inverse is obtained from $\W_t$ using the Sherman--Morrison formula for the upper-left block followed by the standard block-inverse formula\footnote{For $\mathbf{M}=\begin{pmatrix}\mathbf{A} & \mathbf{b}\\ \mathbf{b}^{\top} & d\end{pmatrix}$, $\mathbf{A}\in\reals^{r\times r}$, $\mathbf{b}\in\reals^r$, $d\in\reals$, where $\mathbf{A}$ is invertible and $s:=d-\mathbf{b}^{\top}\mathbf{A}^{-1}\mathbf{b}\neq0$, the standard block-inverse formula gives $\mathbf{M}^{-1}=\begin{pmatrix}\mathbf{A}^{-1}+\mathbf{A}^{-1}\mathbf{b}s^{-1}\mathbf{b}^{\top}\mathbf{A}^{-1} & -\mathbf{A}^{-1}\mathbf{b}s^{-1}\\ -s^{-1}\mathbf{b}^{\top}\mathbf{A}^{-1} & s^{-1}\end{pmatrix}$.}. Both operations require $O(r_t^2)$ arithmetic operations. If $\theta=1$, simply set $\U_{t+1}=\v_{t,+}$ and $\mathbf{S}_{t+1}=\W_{t+1}=1$. The case $\theta=0$ clearly leaves the maintained state unchanged.

Consider next an away step and write $\a:=\U_t^{\top}\v_{t,-}$. If $\theta:=\theta_t^{\rm A}<1$, then the image is unchanged, $\U_{t+1}=\U_t$ and $\mathbf{S}_{t+1}=(1+\theta s_t)\mathbf{S}_t-\theta s_t\a\a^{\top}$. The matrix $\mathbf{S}_{t+1}$ is positive definite and its inverse is obtained from $\W_t$ by rescaling and the Sherman--Morrison formula.

If $\theta_t^{\rm A}=1$, the step is a drop step and the matrix $\widehat{\mathbf{S}}:=(1+s_t)\mathbf{S}_t-s_t\a\a^{\top}$ has rank $r_t-1$ by Lemma \ref{lem:stepsize}. Moreover, by the definition of $s_t$, $\widehat{\mathbf{S}}\W_t\a=0$. Let $\mathbf{H}$ be a Householder reflector satisfying $\mathbf{H}\W_t\a/\Vert{\W_t\a}\Vert_2=\e_{r_t}$. Then the last row and column of $\mathbf{H}^{\top}\widehat{\mathbf{S}}\mathbf{H}$ are zero. We obtain $\U_{t+1}$ by deleting the last column of $\U_t\mathbf{H}$, and $\mathbf{S}_{t+1}$ by deleting the last row and column of $\mathbf{H}^{\top}\widehat{\mathbf{S}}\mathbf{H}$. Moreover, $\W_{t+1}$ is obtained simply by deleting the last row and column of $(1+s_t)^{-1}\mathbf{H}^{\top}\W_t\mathbf{H}$. To verify the last claim, note that $\widehat{\mathbf{S}}\W_t=(1+s_t)\I-s_t\a\a^{\top}\W_t$, and, after the change of coordinates by $\mathbf{H}$, the vector $\W_t\a$ is supported only on the last coordinate. Thus, the product of the leading $(r_t-1)\times(r_t-1)$ blocks of $\mathbf{H}^{\top}\widehat{\mathbf{S}}\mathbf{H}$ and $\mathbf{H}^{\top}\W_t\mathbf{H}$ equals $(1+s_t)\I$.

Finally, suppose that a pairwise step is accepted. Write $\a:=\U_t^{\top}\u_t$. The subtraction part gives
\begin{align*}
\X_t-\gamma_t\u_t\u_t^{\top}
=
\U_t\left(\mathbf{S}_t-\gamma_t\a\a^{\top}\right)\U_t^{\top}.
\end{align*}
By Lemma \ref{lem:stepsize}, the matrix $\mathbf{S}_t-\gamma_t\a\a^{\top}$ has rank $r_t-1$. Moreover, by the definition of $\gamma_t$, $(\mathbf{S}_t-\gamma_t\a\a^{\top})\W_t\a=0$. Applying again a Householder reflector $\mathbf{H}$ which maps $\W_t\a/\Vert{\W_t\a}\Vert_2$ to $\e_{r_t}$ and deleting the last coordinate gives a rank-$(r_t-1)$ representation. In this case the inverse of the reduced matrix is obtained by deleting the last row and column of $\mathbf{H}^{\top}\W_t\mathbf{H}$. Indeed, $(\mathbf{S}_t-\gamma_t\a\a^{\top})\W_t=\I-\gamma_t\a\a^{\top}\W_t$, and the same argument as above shows that the leading blocks after the change of coordinates are mutual inverses.

Let $\U$, $\mathbf{S}$, and $\W=\mathbf{S}^{-1}$ denote this reduced representation, and write $\b:=\U^{\top}\v_t(\theta_t^{\rm P})$ and $\z:=\v_t(\theta_t^{\rm P})-\U\b$. The subsequent addition of $\gamma_t\v_t(\theta_t^{\rm P})\v_t(\theta_t^{\rm P})^{\top}$ is handled exactly as a rank-one addition above. If $\z=0$, then $\U_{t+1}=\U$ and $\mathbf{S}_{t+1}=\mathbf{S}+\gamma_t\b\b^{\top}$, and its inverse is obtained from $\W$ by Sherman--Morrison. If $\z\neq0$, set
\begin{align*}
\U_{t+1}
=
\left[\U,\frac{\z}{\Vert{\z}\Vert_2}\right], \qquad \mathbf{S}_{t+1}
=
\begin{pmatrix}
\mathbf{S}+\gamma_t\b\b^{\top}
&
\gamma_t\Vert{\z}\Vert_2\b
\\
\gamma_t\Vert{\z}\Vert_2\b^{\top}
&
\gamma_t\Vert{\z}\Vert_2^2
\end{pmatrix}.
\end{align*}
Its inverse is obtained from $\W$ by Sherman--Morrison for the upper-left block followed by the standard block-inverse formula. Thus, a pairwise step removes one direction from the current image and then possibly adds one new direction.

All changes of basis above use a single orthogonal transformation and at most a one-dimensional extension or deletion, and all updates of the small matrices use rank-one or block-inverse formulas. Recall that a Householder reflector has the form $\mathbf{H}=\I-2\q\q^{\top}$ for a unit vector $\q$. Hence, $\U_t\mathbf{H}=\U_t-2(\U_t\q)\q^{\top}$ can be formed in $O(nr_t)$ arithmetic operations, while applying the corresponding change of coordinates to $\mathbf{S}_t$ and $\W_t$ costs $O(r_t^2)$ operations. Consequently, updating $\U_t$, $\mathbf{S}_t$, and $\W_t$ requires $O(nr_t+r_t^2)$ arithmetic operations and $O(nr_t+r_t^2)$ storage.

The discussion above concerns the maintained low-rank representation of the iterate, whose memory requirement is $O(nr_t)$. When the gradient is formed and stored as a dense matrix, it is often simplest to maintain in addition a dense copy of $\X_t$. This copy is updated from the accepted rank-one or rank-two step in $O(n^2)$ time. It also allows $\langle{\widetilde{\nabla}_t,\X_t}\rangle$ and each dense trial matrix required by the Armijo searches to be computed in $O(n^2)$ time. Since storing a dense gradient already requires $O(n^2)$ memory, the additional dense copy of $\X_t$ does not change the asymptotic memory requirement. If instead the objective and gradient oracles operate directly on low-rank representations, the dense copy may be omitted. 

\subsection{Why is the pairwise step needed?}
\label{subsec:pairwise-need}

We give a simple example illustrating why, unlike standard linearly-convergent
Frank-Wolfe methods over polytopes  \cite{garber2016linearly, lacoste2015global}, the combination of Frank-Wolfe and away
steps alone is not sufficient for obtaining the desired linear decrease.

Consider the Euclidean projection problem with $n=3$
\[
     f(\X)
    :=
    \frac{1}{2}\|\X-\C\|_F^2,
    \qquad
    \C :=
    \begin{pmatrix}
        1 & 0 & 0\\
        0 & 0 & 0\\
        0 & 0 & -1
    \end{pmatrix}.
\]
The unique optimal solution is clearly  $\X^*=\mathbf{e}_1\mathbf{e}_1^\top$. Indeed,
\[
    \nabla f(\X^*)=
    \begin{pmatrix}
        0&0&0\\
        0&0&0\\
        0&0&1
    \end{pmatrix},
\]
and therefore the first-order optimality condition holds over $\mS^n$. Notice
also that the smallest eigenvalue of $\nabla f(\X^*)$ has multiplicity two.

For any sufficiently small $\varepsilon>0$, consider the unit vector
\[
    \v_\varepsilon
    :=
    \begin{pmatrix}
        \sqrt{1-\varepsilon/2-\varepsilon^2/2}\\
        \varepsilon\\
        \sqrt{(\varepsilon-\varepsilon^2)/2}
    \end{pmatrix},
    \qquad
    \X_\varepsilon:=\v_\varepsilon\v_\varepsilon^\top.
\]
For any unit vector $\v=(v_1,v_2,v_3)^\top$,
$f(\v\v^\top)-f(\X^*)=v_2^2+2v_3^2$,
and hence
$f(\X_\varepsilon)-f(\X^*)=\varepsilon$.

Note that since $\X_\varepsilon$ is rank-one, an away step is not possible.
Consider the Frank-Wolfe step. Let
$\nabla_\varepsilon
    :=
    \nabla f(\X_\varepsilon)
    =
    \X_\varepsilon-\C$. 
A direct calculation gives
$    \v_\varepsilon^\top
    \nabla_\varepsilon
    \v_\varepsilon
    =\varepsilon$,
while the characteristic polynomial of $\nabla_\varepsilon$ is
$\lambda^3-\lambda^2-\varepsilon\lambda+\varepsilon^2$. 
Consequently, writing $\varphi=(1+\sqrt{5})/2$, we have that
$\lambda_3(\nabla_\varepsilon)
    =
    -\varphi\varepsilon+O(\varepsilon^2)$.

Thus, with $\widetilde{\nabla}_\varepsilon
    :=
    \nabla_\varepsilon
    -
    \lambda_3(\nabla_\varepsilon)\I$,
the Frank-Wolfe gap satisfies
$
    \langle \widetilde{\nabla}_\varepsilon,\X_\varepsilon\rangle^2
    =
    \Theta(\varepsilon^2)$.
Moreover, if
$\w_\varepsilon$ denotes a bottom eigenvector of $\nabla_\varepsilon$, then
$\w_\varepsilon$ converges, up to normalization and sign, to
\[
    \mathbf{e}_1-\varphi^{-1}\mathbf{e}_2,
\]
and in particular,
\[
    \|\w_\varepsilon\w_\varepsilon^\top-\X_\varepsilon\|_F
    =\Theta(1).
\]
Set $\D_\varepsilon:=\w_\varepsilon\w_\varepsilon^\top-\X_\varepsilon$. Since $\w_\varepsilon$ is a bottom eigenvector of $\nabla_\varepsilon$ and $\widetilde\nabla_\varepsilon\succeq0$,
\[
-\langle\nabla_\varepsilon,\D_\varepsilon\rangle
=\langle\widetilde\nabla_\varepsilon,\X_\varepsilon\rangle
=\Theta(\varepsilon).
\]
Since $f$ is quadratic,
\[
f(\X_\varepsilon+\theta\D_\varepsilon)
=f(\X_\varepsilon)
+\theta\langle\nabla_\varepsilon,\D_\varepsilon\rangle
+\frac{\theta^2}{2}\Vert{\D_\varepsilon}\Vert_F^2.
\]
Hence, for all sufficiently small $\varepsilon$, the exact line-search minimizer is in $(0,1)$ and equals
\[
\theta_\varepsilon^*=
\frac{-\langle\nabla_\varepsilon,\D_\varepsilon\rangle}
{\Vert{\D_\varepsilon}\Vert_F^2}
=\Theta(\varepsilon),
\]
which gives
\[
\begin{split}
    f(\X_\varepsilon)
    -
    \min_{\theta\in[0,1]}
    f\bigl(
        \X_\varepsilon
        +
        \theta(
            \w_\varepsilon\w_\varepsilon^\top-\X_\varepsilon
        )
    \bigr)
    =
    \Theta(\varepsilon^2).
\end{split}
\]
Hence neither the Frank-Wolfe step nor the away step can provide a decrease of
order $\varepsilon$, and thus they are insufficient for obtaining a per-iteration linear decrease. As we shall see in the sequel, introducing the additional pairwise step resolves this issue.

\section{Convergence Rate Analysis}\label{sec:analysis}

Throughout this section, denote $h_t:=f(\X_t)-f^*$ for all $t\geq 1$.

\begin{theorem}\label{thm:main}
Let $(\X_t)_{t\geq1}$ be the sequence produced by Algorithm \ref{alg:newFW}. Without Assumptions \ref{ass:qg} and \ref{ass:optimal-rank}, for every $t\geq2$,
\begin{align}\label{eq:global-sublinear-rate}
h_t
\leq
\frac{8\beta{D}^2}{t+7}.
\end{align}
If, in addition, Assumptions \ref{ass:qg} and \ref{ass:optimal-rank} hold, define
\begin{align}\label{eq:rho}
\rho
:={}
\min\left\{
\frac{1}{2},
\frac{(\sqrt{2}-1)\alpha\lambda_{r^*}^*}
{2560r^*\kappa_{r^*}^2\left({\frac{\beta{D}^2}{2}+\widetilde G}\right)}
\right\},
\end{align}
Then, on every iteration $t$ which is not a drop iteration,
\begin{align}\label{eq:productive-contraction}
h_{t+1}\leq(1-\rho)h_t.
\end{align}
Finally, for every $t\geq1$,
\begin{align}\label{eq:global-linear-rate}
h_t
\leq
h_1(1-\rho)^{\lceil(t-1)/2\rceil}.
\end{align}
\end{theorem}

The sublinear rate \eqref{eq:global-sublinear-rate} is of the same order as the standard Frank-Wolfe, i.e., $\Theta(\beta{D}^2/t)$. For the Frobenius norm, $\Vert{\cdot}\Vert=\Vert{\cdot}\Vert_*=\Vert{\cdot}\Vert_F$,  ${D}=\sqrt{2}$ and $\kappa_{r^*}=1$, so the denominator in the second term of \eqref{eq:rho} is simply $r^*(\beta+\widetilde G)$. If the primal norm is the nuclear norm, so that $\Vert{\cdot}\Vert_*=\Vert{\cdot}\Vert_2$ is the spectral norm, then ${D}=2$ and again $\kappa_{r^*}=1$, yielding again $r^*(2\beta+\widetilde G)$ in the corresponding denominator. 

\subsection{Proof of Theorem \ref{thm:main}}
The proof is structured as follows.
\begin{itemize}
\item
Lemma \ref{lem:support-armijo} gives a standard argument that the objective decrease for the FW or away steps is proportional to the quantities $a_t$ and $b_t$ (defined in \eqref{eq:algorithm-certificates}), respectively.
\item
Lemma \ref{lem:pairwise-decrease} develops the same type of objective decrease argument, with respect to the certificate $c_t$ (defined in \eqref{eq:algorithm-certificates}), for the pairwise step. In particular, the proof guides the reader as to why we define the pairwise step as we do.  
\item
Lemma \ref{lem:certificate-lower} is the central lemma that establishes that indeed one of the three certificates (in case of a non-drop-step) must result in sufficient objective decrease, i.e., $\max\{a_t,b_t,c_t\} = \Omega(h_t)$.
\item
Finally, the proof of Theorem \ref{thm:main} combines the above ingredients to yield the linear convergence rate. The sublinear rate follows from standard arguments and mostly mirrors the standard Frank-Wolfe analysis. 
\end{itemize}
Throughout this subsection, we use the notation of Algorithm \ref{alg:newFW} and set
\begin{align}\label{eq:fixed-notation}
L:=\frac{\beta{D}^2}{2}+\widetilde G.
\end{align}
Note that since $\X_t\succeq0$ and $\trace(\X_t)=1$, and by the definition of $\v_{t,-}$,
\begin{align}\label{eq:support-bounds-upper}
0
\leq
\langle{\widetilde{\nabla}_t,\X_t}\rangle
\leq
\v_{t,-}^{\top}\widetilde{\nabla}_t\v_{t,-}
\leq
\Vert{\widetilde{\nabla}_t}\Vert_2
\leq
\widetilde G.
\end{align}

\begin{lemma}[objective decrease for FW / away steps]\label{lem:support-armijo}
The Frank-Wolfe Armijo search terminates and
\begin{align}\label{eq:fw-armijo-decrease}
f(\X_t)-f(\X_t^{\rm FW})
\geq
\frac{\langle{\widetilde{\nabla}_t,\X_t}\rangle^2}{8L}.
\end{align}
The away Armijo search terminates. If the full step is accepted ($\theta_t^{\rm{A}} = 1$) it is a drop step; otherwise, 
\begin{align}\label{eq:away-armijo-decrease}
f(\X_t)-f(\X_t^{\rm A})
\geq
\frac{
\left({
\v_{t,-}^{\top}\widetilde{\nabla}_t\v_{t,-}
-
\langle{\widetilde{\nabla}_t,\X_t}\rangle
}\right)^2
}{8L}.
\end{align}
\end{lemma}

\begin{proof}
For the Frank-Wolfe direction $\D=\v_{t,+}\v_{t,+}^{\top}-\X_t$, we have $\trace(\D)=0$ and hence
\begin{align*}
\langle{\nabla_t,\D}\rangle
=
\langle{\widetilde{\nabla}_t,\D}\rangle
=-\langle{\widetilde{\nabla}_t,\X_t}\rangle,
\qquad
\Vert{\D}\Vert\leq{D}.
\end{align*}
Thus, by the smoothness of $f$,
\begin{align*}
f(\X_t+\theta\D)
\leq
f(\X_t)
-\theta\langle{\widetilde{\nabla}_t,\X_t}\rangle
+\frac{\beta{D}^2}{2}\theta^2.
\end{align*}
Every $0<\theta\leq \langle{\widetilde{\nabla}_t,\X_t}\rangle/(\beta{D}^2)$
satisfies the Armijo condition \eqref{eq:armijoCond}. If the full step is not accepted, the dyadic backtracking therefore returns $\theta> \langle{\widetilde{\nabla}_t,\X_t}\rangle/(2\beta{D}^2)$,
and \eqref{eq:fw-armijo-decrease} follows since $L\geq\beta{D}^2/2$. If the full step is accepted, the same bound follows from \eqref{eq:support-bounds-upper} and $\widetilde G\leq L$.

For the away step, if $b_t=0$ or the path is degenerate, the claim follows immediately from the Armijo condition \eqref{eq:armijoCond} and the conventions above. Otherwise, let
\begin{align*}
\D
:=
\frac{\X_t-\v_{t,-}\v_{t,-}^{\top}}
{\v_{t,-}^{\top}\X_t^{\dagger}\v_{t,-}-1}.
\end{align*}
The direction $\D$ also has zero trace, and therefore
\begin{align*}
\langle{\nabla_t,\D}\rangle
&=
\langle{\widetilde{\nabla}_t,\D}\rangle
=-
\frac{
\v_{t,-}^{\top}\widetilde{\nabla}_t\v_{t,-}
-
\langle{\widetilde{\nabla}_t,\X_t}\rangle
}{
\v_{t,-}^{\top}\X_t^{\dagger}\v_{t,-}-1
}, \quad 
\Vert{\D}\Vert
\leq
\frac{{D}}
{\v_{t,-}^{\top}\X_t^{\dagger}\v_{t,-}-1}.
\end{align*}

Hence, using again the smoothness of $f$,
\begin{align*}
f(\X_t+\theta\D)
\leq
f(\X_t)
&-
\theta
\frac{
\v_{t,-}^{\top}\widetilde{\nabla}_t\v_{t,-}
-
\langle{\widetilde{\nabla}_t,\X_t}\rangle
}{
\v_{t,-}^{\top}\X_t^{\dagger}\v_{t,-}-1
}
+
\frac{\beta{D}^2\theta^2}{2}
\frac{1}
{\left({\v_{t,-}^{\top}\X_t^{\dagger}\v_{t,-}-1}\right)^2}.
\end{align*}
The same dyadic argument gives \eqref{eq:away-armijo-decrease} whenever the full step is not accepted. If the full step is accepted, it is a drop step.
\end{proof}

\begin{lemma}[objective decrease for pairwise step]\label{lem:pairwise-decrease}
Let $\u\in\textrm{Im}(\X_t)$ be a unit vector such that $(\I-\u\u^{\top})\widetilde\nabla_t\u \neq \mathbf{0}$
and consider the pairwise  Armijo search defined in \eqref{eq:armijoCond} (w.r.t. $\Phi_t^{\rm{P}}(\theta)$ with $\v_t(\theta)$ and $\p_t$ defined according to $\u$ in place of $\u_t$). Denoting its output by $\X_t^{\rm P}$, we have that
\begin{align}\label{eq:pairwise-decrease}
f(\X_t)-f(\X_t^{\rm P})
\geq
\frac{1}{4L}
\frac{\u^{\top}\widetilde{\nabla}_t^2\u-(\u^{\top}\widetilde{\nabla}_t\u)^2}
{\u^{\top}\X_t^{\dagger}\u}.
\end{align}
\end{lemma}

\begin{proof}
In the following set $\gamma = (\u^{\top}\X_t^{\dagger}\u)^{-1}$ and let $\p$ be some unit vector orthogonal to $\u$ (indeed in Algorithm \ref{alg:newFW}, we have $\p_t \perp \u_t$). For any $\theta\in[0,1]$, define $\v(\theta):=\frac{\u-\theta\p}{\sqrt{1+\theta^2}}$.
Let $\theta\in(0,1]$ and denote $\D := \v(\theta)\v(\theta)^{\top} - \u\u^{\top}$. Note that $\D$ is rank-two and that $\trace(\D) =0$, and so its two nonzero eigenvalues are given by $\pm\lambda$, for some $\lambda > 0$. In particular,
\begin{align*}
2\lambda^2 = \Vert{\D}\Vert_F^2 = 2(1-(\v(\theta)^{\top}\u)^2) = 2\left({1-\frac{1}{1+\theta^2}}\right) = 2\frac{\theta^2}{1+\theta^2}.
\end{align*}
Thus,  $\lambda = \theta/\sqrt{1+\theta^2}$ and, by unitary invariance and the definition of ${D}$, $\Vert{\D}\Vert
\leq
\frac{\theta}{\sqrt{1+\theta^2}}D$.

By the smoothness of $f$ we have that,
\begin{align}\label{eq:lem:pw:1}
f(\X_t + \gamma\D) &\leq f(\X_t) + \gamma\langle{\D, \nabla_t}\rangle + \frac{\beta{}D^2\gamma^2\theta^2}{2(1+\theta^2)} \nonumber \\
&= f(\X_t) + \gamma\langle{\D, \widetilde\nabla_t}\rangle + \frac{\beta{}D^2\gamma^2\theta^2}{2(1+\theta^2)},
\end{align}
where the equality follows since $\trace(\D) = 0$.

Write now $\D$ as:
\begin{align*}
\D =  \v(\theta)\v(\theta)^{\top} - \u\u^{\top} &= \frac{1}{1+\theta^2}\left({\u\u^{\top} + \theta^2\p\p^{\top} - \theta\p\u^{\top} - \theta\u\p^{\top}}\right) - \u\u^{\top} \\
&= \frac{1}{1+\theta^2}\left({\theta^2\p\p^{\top} - \theta^2\u\u^{\top} - \theta\p\u^{\top} - \theta\u\p^{\top}}\right). 
\end{align*}
This gives,
\begin{align}\label{eq:lem:pw:2}
\langle{\D, \widetilde\nabla_t}\rangle &= \frac{1}{1+\theta^2}\left({-2\theta\p^{\top}\widetilde\nabla_t\u + \theta^2\p^{\top}\widetilde\nabla_t\p - \theta^2\u^{\top}\widetilde\nabla_t\u}\right) \nonumber \\
&\leq  -\frac{2\theta}{1+\theta^2}\p^{\top}\widetilde\nabla_t\u + \frac{\theta^2}{1+\theta^2}\widetilde{G},
\end{align}
where the inequality uses the fact that $\widetilde\nabla_t \succeq 0$ and the definition of $\widetilde{G}$.

Note that the RHS of \eqref{eq:lem:pw:2} readily explains the choice of the vector $\p$, given the vector $\u$: among the unit vectors orthogonal to $\u$, choose the one that maximizes the inner product $\p^{\top}(\widetilde\nabla_t\u)$, which in turn minimizes the upper-bound on $\langle{\D, \widetilde\nabla_t}\rangle$ (which goes into the function value reduction argument in \eqref{eq:lem:pw:1}). Indeed, for a given $\u$, this maximizer is given by the value
\begin{align}\label{eq:lem:pw:3}
\p:=
\frac{(\I-\u\u^{\top})\widetilde\nabla_t\u}
{\Vert{(\I-\u\u^{\top})\widetilde\nabla_t\u}\Vert_2},
\end{align}
which is exactly the choice of $\p_t$ in our Algorithm \ref{alg:newFW} when using $\u$ in place of $\u_t$.

Assume for the rest of the proof that $\p$ is indeed given by \eqref{eq:lem:pw:3} and note that with this choice,
\begin{align*}
\langle{\D, \widetilde\nabla_t}\rangle &\leq -\frac{2\theta}{1+\theta^2}\Vert{(\I-\u\u^{\top})\widetilde\nabla_t\u}\Vert_2 +  \frac{\theta^2}{1+\theta^2}\widetilde{G},
\end{align*}
and plugging-back into \eqref{eq:lem:pw:1}, and using the fact that $\gamma^2 \leq \gamma \leq 1$ we have that,
\begin{align*}
f(\X_t + \gamma\D) &\leq f(\X_t) -\frac{2\theta\gamma}{1+\theta^2}\Vert{(\I-\u\u^{\top})\widetilde\nabla_t\u}\Vert_2 + \frac{L\gamma\theta^2}{1+\theta^2}.
\end{align*}

Since $\gamma \leq 1$ and $\Vert{(\I-\u\u^{\top})\widetilde\nabla_t\u}\Vert_2 \leq \Vert{\widetilde{\nabla}_t}\Vert_2 \leq L$,
every $0<\theta\leq \Vert{(\I-\u\u^{\top})\widetilde\nabla_t\u}\Vert_2/(2L)$
satisfies
\begin{align*}
f\left({\X_t+\gamma\D}\right)
\leq
f(\X_t)
-\theta\gamma\Vert{(\I-\u\u^{\top})\widetilde\nabla_t\u}\Vert_2.
\end{align*}
Thus, dyadic backtracking returns $\theta>
\Vert{(\I-\u\u^{\top})\widetilde\nabla_t\u}\Vert_2/(4L)$, which gives
\begin{align}\label{eq:lem:pw:4}
f(\X_t)-f(\X_t+\gamma\D)
&>
\frac{\gamma}{4L}
\Vert{(\I-\u\u^{\top})\widetilde\nabla_t\u}\Vert_2^2= \frac{1}{4L}
\frac{\u^{\top}\widetilde{\nabla}_t^2\u-(\u^{\top}\widetilde{\nabla}_t\u)^2}
{\u^{\top}\X_t^{\dagger}\u},
\end{align}
where we recall that $\gamma = (\u^{\top}\X_t^{\dagger}\u)^{-1}$.

This proves \eqref{eq:pairwise-decrease}.

Observe that the RHS of \eqref{eq:lem:pw:4} explains the choice of $\u_t$ in Algorithm \ref{alg:newFW}. Maximizing the entire RHS as a function of such $\u$ does not seem to admit a simple solution, but as an approximation (which indeed works well enough), we settle for maximizing only the first term, i.e., taking
\begin{align*}
\u\in\argmax_{\w\in\rm{Im}(\X_t):\Vert{\w}\Vert_2=1}\frac{\w^{\top}\widetilde{\nabla}_t^2\w}
{\w^{\top}\X_t^{\dagger}\w},
\end{align*}
which corresponds precisely to the choice of $\u_t$ in the algorithm.
\end{proof}

The following auxiliary lemma is required for the proof of the lemma following it. A proof is given in the appendix for completeness. 
\begin{lemma}[see also Lemma 5.4 in \cite{tu2016low}]\label{lem:factor-comparison}
Let $\Y,\Z\in\reals^{n\times r}$, and suppose $\rank(\Z)=r$. Then,
\begin{align}\label{eq:factor-comparison}
\min_{\bR\in\reals^{r\times r}:\bR^{\top}\bR=\I}
\Vert{\Y-\Z\bR}\Vert_F^2
\leq
\frac{1}{2(\sqrt{2}-1)\sigma_r(\Z)^2}
\Vert{\Y\Y^{\top}-\Z\Z^{\top}}\Vert_F^2,
\end{align}
where $\sigma_r(\Z)$ denotes the smallest singular value of $\Z$.
\end{lemma}

The following lemma is the central replacement for the strict-complementarity argument. It directly lower bounds the largest of the three quantities associated with the steps in Algorithm \ref{alg:newFW}.

\begin{lemma}[certificate lower bound]\label{lem:certificate-lower}
\begin{align}\label{eq:certificate-lower}
\max\Bigg\{
&\langle{\widetilde{\nabla}_t,\X_t}\rangle^2, ~\left({
\v_{t,-}^{\top}\widetilde{\nabla}_t\v_{t,-}
-
\langle{\widetilde{\nabla}_t,\X_t}\rangle
}\right)^2,~
\frac{\u_t^\top\widetilde{\nabla}_t^2\u_t-(\u_t^\top\widetilde{\nabla}_t\u_t)^2}
{\u_t^\top\X_t^\dagger\u_t}
\Bigg\} \nonumber \\
&\geq
\frac{(\sqrt{2}-1)\lambda_{r^*}^*\alpha}{320r^*\kappa_{r^*}^2}h_t.
\end{align}
\end{lemma}

\begin{proof}
If $h_t=0$, the claim is immediate. Denote
$\X^*\in\argmin_{\Y\in\mX^*}\Vert{\X_t-\Y}\Vert$ and  $d:=\Vert{\X_t-\X^*}\Vert$.
Recall Assumptions \ref{ass:qg} and \ref{ass:optimal-rank} give
\begin{align}\label{eq:closest-qg}
d^2
\leq
\frac{2h_t}{\alpha},
\qquad
\rank(\X^*)=r^*,
\qquad
\lambda_{r^*}(\X^*)\geq\lambda_{r^*}^*.
\end{align}

Let $\X_t=\sum_{i=1}^n\lambda_i(\X_t)\q_i\q_i^{\top}$
be an eigen-decomposition of $\X_t$ and define
\begin{align}\label{eq:Xr-E}
\X_{t,r^*}
&:=
\sum_{i=1}^{r^*}\lambda_i(\X_t)\q_i\q_i^{\top},
&
\matE&:=\X_t-\X_{t,r^*},
&
\tau&:=\trace(\matE).
\end{align}

Let $\P^*$ denote the projection matrix onto $\textrm{Im}(\X^*)$. By Ky Fan's maximum principle,
\begin{align}
\tau
&=
1-\sum_{i=1}^{r^*}\lambda_i(\X_t)
\leq
1-\trace(\P^*\X_t)
=
\trace\left({\P^*(\X^*-\X_t)}\right)
\nonumber\\
&\leq
\Vert{\P^*}\Vert_*d
\leq
\sqrt{r^*}\,\kappa_{r^*}d,
\label{eq:tail-trace}
\end{align}
where the last inequality follows since for any $\A$ with $\Vert{\A}\Vert\leq1$,
\begin{align*}
\left|{\langle{\P^*,\A}\rangle}\right|
&=
\left|{\trace(\P^*\A\P^*)}\right|
\leq
\sqrt{r^*}\Vert{\P^*\A\P^*}\Vert_F
\\
&\leq
\sqrt{r^*}\,\kappa_{r^*}\Vert{\P^*\A\P^*}\Vert
\leq
\sqrt{r^*}\,\kappa_{r^*}.
\end{align*}
Also, since $\rank(\X^*)=r^*$, the best rank-$r^*$ approximation property for unitarily invariant norms gives
\begin{align}\label{eq:Xr-distance}
\Vert{\X_t-\X_{t,r^*}}\Vert
&\leq d,
&
\Vert{\X_{t,r^*}-\X^*}\Vert
&\leq2d,
&
\Vert{\X_{t,r^*}-\X^*}\Vert_F
&\leq2\kappa_{r^*}d.
\end{align}

Write
\begin{align*}
\X_{t,r^*}=\Q\bLambda\Q^{\top},
\qquad
\Y:=\Q\bLambda^{1/2},
\qquad
\X^*=\Z\Z^{\top},
\end{align*}
where $\Q=[\q_1,\ldots,\q_{r^*}]$ and $\bLambda$ is the diagonal matrix containing the first $r^*$ eigenvalues of $\X_t$. Align $\Z$ with $\Y$ by an optimal right orthogonal transformation (in the sense of Lemma \ref{lem:factor-comparison}), and set
$\D:=\Y-\Z$.
By convexity and the equality $\trace(\X_t)=\trace(\X^*)=1$,
\begin{align}\label{eq:convexity-B}
h_t
\leq
\langle{\nabla_t,\X_t-\X^*}\rangle
=
\langle{\widetilde{\nabla}_t,\X_t-\X^*}\rangle.
\end{align}
Using
$\X_{t,r^*}-\X^*
=
\Y\D^{\top}+\D\Y^{\top}-\D\D^{\top}$ 
and $\widetilde{\nabla}_t\succeq0$, we obtain
\begin{align}
\langle{\widetilde{\nabla}_t,\X_{t,r^*}-\X^*}\rangle
&=
2\langle{\widetilde{\nabla}_t\Y,\D}\rangle
-
\langle{\widetilde{\nabla}_t,\D\D^{\top}}\rangle
\leq
2\langle{\widetilde{\nabla}_t\Y,\D}\rangle.
\label{eq:factor-linearization}
\end{align}
Consequently,
\begin{align}\label{eq:key-decomposition}
h_t
\leq
\langle{\widetilde{\nabla}_t,\X_t-\X_{t,r^*}}\rangle
+
2\langle{\widetilde{\nabla}_t\Y,\D}\rangle.
\end{align}

We consider now two cases. Suppose first that $\langle{\widetilde{\nabla}_t,\X_t-\X_{t,r^*}}\rangle \geq \frac{h_t}{2}$.
The matrix $\matE$ is positive semidefinite and is supported on $\textrm{Im}(\X_t)$. Therefore,
\begin{align*}
\langle{\widetilde{\nabla}_t,\X_t-\X_{t,r^*}}\rangle
=
\trace(\widetilde{\nabla}_t\matE)
&\leq
\left({
\max_{\v\in\textrm{Im}(\X_t):\Vert{\v}\Vert_2=1}
\v^{\top}\widetilde{\nabla}_t\v
}\right)
\tau.
\end{align*}
Using \eqref{eq:tail-trace} and \eqref{eq:closest-qg},
\begin{align*}
\left({
\max_{\v\in\textrm{Im}(\X_t):\Vert{\v}\Vert_2=1}
\v^{\top}\widetilde{\nabla}_t\v
}\right)^2
\geq
\frac{\alpha}{8r^*\kappa_{r^*}^2}h_t.
\end{align*}
Thus, by the definition of $\v_{t,-}$, and using $\max\{a^2,(a-b)^2\} \geq b^2/4$ for any two non-negative scalars $a,b$, we have that
\begin{align*}
\max\Big\{
&\langle{\widetilde{\nabla}_t,\X_t}\rangle^2,
\left({
\v_{t,-}^{\top}\widetilde{\nabla}_t\v_{t,-}
-
\langle{\widetilde{\nabla}_t,\X_t}\rangle
}\right)^2
\Big\}
\geq
\frac{1}{4}\left({\v_{t,-}^{\top}\widetilde{\nabla}_t\v_{t,-}}\right)^2
\geq
\frac{\alpha}{32r^*\kappa_{r^*}^2}h_t,
\end{align*}
which proves \eqref{eq:certificate-lower} in this case since $\lambda_{r^*}^*\leq1$.

Suppose now that $\langle{\widetilde{\nabla}_t,\X_t-\X_{t,r^*}}\rangle < \frac{h_t}{2}$.
From \eqref{eq:key-decomposition},
\begin{align}\label{eq:signal-factor-inner-product}
2\langle{\widetilde{\nabla}_t\Y,\D}\rangle
>
\frac{h_t}{2}.
\end{align}
We have
\begin{align}\label{eq:weighted-trace-identity}
\trace(\widetilde{\nabla}_t^2\X_{t,r^*})
=
\Vert{\widetilde{\nabla}_t\Y}\Vert_F^2.
\end{align}
Since $\sigma_{r^*}(\Z)^2=\lambda_{r^*}(\X^*)\geq\lambda_{r^*}^*$, Lemma \ref{lem:factor-comparison} and \eqref{eq:Xr-distance} give
\begin{align}\label{eq:unweighted-D-bound}
\Vert{\D}\Vert_F^2
\leq
\frac{1}{2(\sqrt{2}-1)\lambda_{r^*}(\X^*)}
\Vert{\X_{t,r^*}-\X^*}\Vert_F^2
\leq
\frac{2\kappa_{r^*}^2}{(\sqrt{2}-1)\lambda_{r^*}^*}d^2.
\end{align}
Combining $\langle{\widetilde{\nabla}_t\Y,\D}\rangle
\leq
\Vert{\widetilde{\nabla}_t\Y}\Vert_F\Vert{\D}\Vert_F$ with \eqref{eq:signal-factor-inner-product}, \eqref{eq:weighted-trace-identity}, and \eqref{eq:unweighted-D-bound} yields
\begin{align*}
\trace(\widetilde{\nabla}_t^2\X_{t,r^*})
&=
\Vert{\widetilde{\nabla}_t\Y}\Vert_F^2
>
\frac{h_t^2}{16\Vert{\D}\Vert_F^2}
\geq
\frac{(\sqrt{2}-1)\lambda_{r^*}^*h_t^2}{32\kappa_{r^*}^2d^2}
\geq
\frac{(\sqrt{2}-1)\alpha\lambda_{r^*}^*}{64\kappa_{r^*}^2}h_t.
\end{align*}
Since $\matE\succeq0$,
\begin{align*}
\widetilde{\nabla}_t\X_t\widetilde{\nabla}_t
=
\widetilde{\nabla}_t\X_{t,r^*}\widetilde{\nabla}_t
+
\widetilde{\nabla}_t\matE\widetilde{\nabla}_t
\succeq
\widetilde{\nabla}_t\X_{t,r^*}\widetilde{\nabla}_t.
\end{align*}
The matrix $\widetilde{\nabla}_t\X_{t,r^*}\widetilde{\nabla}_t$ is positive semidefinite, has rank at most $r^*$, and has trace $\trace(\widetilde{\nabla}_t^2\X_{t,r^*})$. Therefore,
\begin{align}\label{eq:contact-lower}
\lambda_1(\widetilde{\nabla}_t\X_t\widetilde{\nabla}_t)
\geq
\frac{(\sqrt{2}-1)\alpha\lambda_{r^*}^*}{64r^*\kappa_{r^*}^2}h_t.
\end{align}
Directly from the definition of $\u_t$ we have that,
\begin{align}\label{eq:max-ratio}
\frac{\u_t^\top\widetilde{\nabla}_t^2\u_t}{\u_t^\top\X_t^\dagger\u_t}
= 
\lambda_1\left({\X_t^{1/2}\widetilde\nabla_t^2\X_t^{1/2}}\right) = 
\lambda_1(\widetilde{\nabla}_t\X_t\widetilde{\nabla}_t).
\end{align}
Consider now two sub-cases. If 
\begin{align*}
\left({
\max_{\v\in\textrm{Im}(\X_t):\Vert{\v}\Vert_2=1}
\v^{\top}\widetilde{\nabla}_t\v
}\right)^2
\geq
\frac{4}{5}
\frac{\u_t^\top\widetilde{\nabla}_t^2\u_t}{\u_t^\top\X_t^\dagger\u_t},
\end{align*}
then, as before, using the definition of $\v_{t,-}$, we have that
\begin{align*}
\max\Big\{
\langle{\widetilde{\nabla}_t,\X_t}\rangle^2,
\left({
\v_{t,-}^{\top}\widetilde{\nabla}_t\v_{t,-}
-
\langle{\widetilde{\nabla}_t,\X_t}\rangle
}\right)^2
\Big\}
&\geq
\frac{1}{4}\left({\v_{t,-}^{\top}\widetilde{\nabla}_t\v_{t,-}}\right)^2
\\
&\geq
\frac{1}{5}
\frac{\u_t^\top\widetilde{\nabla}_t^2\u_t}{\u_t^\top\X_t^\dagger\u_t} \\
&\geq
\frac{(\sqrt{2}-1)\lambda_{r^*}^*\alpha}{320r^*\kappa_{r^*}^2}h_t.
\end{align*}
Otherwise, 
\begin{align*}
\left({\u_t^\top\widetilde{\nabla}_t\u_t}\right)^2
\leq
\left({
\max_{\v\in\textrm{Im}(\X_t):\Vert{\v}\Vert_2=1}
\v^{\top}\widetilde{\nabla}_t\v
}\right)^2
<
\frac{4}{5}
\frac{\u_t^\top\widetilde{\nabla}_t^2\u_t}{\u_t^\top\X_t^\dagger\u_t}
\leq
\frac{4}{5}\u_t^\top\widetilde{\nabla}_t^2\u_t,
\end{align*}
where the last inequality holds since all positive eigenvalues of $\X_t^{\dagger}$ are at least $1$.

Using \eqref{eq:contact-lower} and \eqref{eq:max-ratio}, it follows that
\begin{align*}
\frac{\u_t^\top\widetilde{\nabla}_t^2\u_t-(\u_t^\top\widetilde{\nabla}_t\u_t)^2}{\u_t^\top\X_t^\dagger\u_t}
&>
\frac{1}{5}
\frac{\u_t^\top\widetilde{\nabla}_t^2\u_t}{\u_t^\top\X_t^\dagger\u_t}
\geq
\frac{(\sqrt{2}-1)\lambda_{r^*}^*\alpha}{320r^*\kappa_{r^*}^2}h_t.
\end{align*}
Thus, \eqref{eq:certificate-lower} holds in all cases.
\end{proof}


\begin{proof}[Proof of Theorem \ref{thm:main}]
Suppose first that Algorithm \ref{alg:newFW} terminates at iteration $t$. Then $\langle{\widetilde{\nabla}_t,\X_t}\rangle=0$,
and by convexity and the definition of $\v_{t,+}$ indeed,
\begin{align*}
h_t
&\leq
\max_{\Y\in\mS^n}
\langle{\nabla_t,\X_t-\Y}\rangle
=
\langle{\widetilde{\nabla}_t,\X_t}\rangle
=0.
\end{align*}
Thus, $\X_t$ is optimal. Before termination, every drop iteration is an Armijo step, while on every non-drop iteration the accepted point is one of the Armijo points. Hence, $h_{t+1}\leq h_t$ on every iteration.

Let $D_t$ denote the number of drop iterations among the first $t-1$ iterations, and let $N_t=t-1-D_t$ denote the number of non-drop iterations. Algorithm \ref{alg:newFW} starts from a rank-one matrix. By Lemma \ref{lem:stepsize}, every drop iteration lowers the rank by one, while every non-drop iteration increases the rank by at most one. For a pairwise step, the maximal subtraction lowers the rank by one and the subsequent rank-one addition increases it by at most one. Therefore,
\begin{align*}
1\leq\rank(\X_t)\leq1+N_t-D_t,
\end{align*}
meaning,
\begin{align}\label{eq:non-drop-count}
D_t\leq N_t,
\qquad
N_t\geq\left\lceil\frac{t-1}{2}\right\rceil.
\end{align}

We first prove the standard bound \eqref{eq:global-sublinear-rate}, without using Assumptions \ref{ass:qg} and \ref{ass:optimal-rank}. Let $C:=\beta{D}^2$. Consider a non-drop iteration $t$. Set $g_t:=\langle{\widetilde{\nabla}_t,\X_t}\rangle$. By convexity and the definition of $\v_{t,+}$, for any $\X^*\in\mX^*$,
\begin{align*}
g_t
\geq
\langle{\nabla_t,\X_t-\X^*}\rangle
\geq
h_t.
\end{align*}
The Frank-Wolfe Armijo point is among the points considered by the algorithm. If its full step is accepted ($\theta_t^{\rm{FW}} = 1$), then from the smoothness of $f$,
\begin{align*}
f(\X_t^{\rm FW})-f^*
\leq
h_t-g_t+\frac{C}{2}
\leq
\frac{C}{2}.
\end{align*}
If its full step is not accepted, then $g_t<C$ and the dyadic argument in the proof of Lemma \ref{lem:support-armijo} gives $\theta_t^{\rm FW}>g_t/(2C)$. Hence, by the Armijo condition and since the accepted iterate is no worse than $\X_t^{\rm FW}$,
\begin{align*}
h_{t+1}
<
h_t-\frac{g_t^2}{4C}
\leq
h_t-\frac{h_t^2}{4C}.
\end{align*}
Consequently, after the first non-drop iteration the objective gap is at most $3C/4$. Thereafter, on every non-drop iteration, the same recurrence holds also when the full Frank-Wolfe step is accepted, since the Armijo condition gives 
\begin{align*}
h_{t+1} \leq h_t-\frac{g_t}{2} \leq h_t - \frac{h_t}{2}  \leq h_t-\frac{h_t^2}{4C},
\end{align*}
where in the last inequality we have used the above observation that $h_t \leq 3C/4$.

Ignoring the intervening drop iterations, which can only decrease the objective, the standard Frank-Wolfe induction for this recurrence gives, after $k\geq1$ non-drop iterations,
\begin{align*}
h_{(k)}\leq\frac{4C}{k+4}.
\end{align*}
Indeed, the bound holds for $k=1$ since $3C/4\leq4C/5$, and, since $x-x^2/(4C)$ is increasing on $[0,2C]$, the induction step is
\begin{align*}
\frac{4C}{k+4}
-\frac{1}{4C}\left({\frac{4C}{k+4}}\right)^2
=
\frac{4C(k+3)}{(k+4)^2}
\leq
\frac{4C}{k+5}.
\end{align*}
Together with monotonicity on drop iterations and \eqref{eq:non-drop-count}, for every $t\geq2$,
\begin{align*}
h_t
\leq
\frac{4\beta{D}^2}{N_t+4}
\leq
\frac{8\beta{D}^2}{t+7},
\end{align*}
which proves \eqref{eq:global-sublinear-rate}.

Suppose now that Assumptions \ref{ass:qg} and \ref{ass:optimal-rank} hold, and fix a non-drop iteration $t$. Since the away search did not accept the full step, Lemmas \ref{lem:support-armijo} and \ref{lem:pairwise-decrease}, together with the rule of taking the best of the three Armijo points, give
\begin{align*}
f(\X_t)-f(\X_{t+1})
\geq
\frac{1}{8L}\max\{a_t,b_t,c_t\}.
\end{align*}
By \eqref{eq:certificate-lower}, it follows that
\begin{align}\label{eq:direct-productive-decrease}
f(\X_t)-f(\X_{t+1})
\geq
\frac{(\sqrt{2}-1)\lambda_{r^*}^*\alpha}
{2560r^*\kappa_{r^*}^2L}h_t,
\end{align}
which by the 
By the definition of $\rho$ in \eqref{eq:rho} gives $h_{t+1} \leq (1-\rho)h_t$,
which proves \eqref{eq:productive-contraction}. 

Taking into account the above upper-bound on the number of drop iterations yields \eqref{eq:global-linear-rate}.
\end{proof}

\section{Acknowledgments}
This work was funded by the European Union (ERC,  ProFreeOpt, 101170791). Views and opinions expressed are however those of the author(s) only and do not necessarily reflect those of the European Union or the European Research Council Executive Agency. Neither the European Union nor the granting authority can be held responsible for them.

\bibliography{bibs.bib}
\bibliographystyle{plain}
\appendix

\section{Proof of Lemma \ref{lem:stepsize}}\label{app:stepsize}

\begin{proof}[Proof of Lemma \ref{lem:stepsize}]
It holds that
\begin{align*}
1-\lambda\v^{\top}\X^{\dagger}\v\geq0
&\Longrightarrow
\I-\lambda{\X^{\dagger}}^{1/2}\v\v^{\top}{\X^{\dagger}}^{1/2}\succeq0
\\
&\Longrightarrow
\X^{1/2}
\left({
\I-\lambda{\X^{\dagger}}^{1/2}\v\v^{\top}{\X^{\dagger}}^{1/2}
}\right)
\X^{1/2}
\succeq0
\\
&\Longrightarrow
\X-\lambda\v\v^{\top}\succeq0,
\end{align*}
where the last implication uses $\v\in\textrm{Im}(\X)$ and hence
$\X^{1/2}{\X^{\dagger}}^{1/2}\v=\v$.

Suppose now that $\lambda=(\v^{\top}\X^{\dagger}\v)^{-1}$. Since $\v\in\textrm{Im}(\X)$, we have
$\textrm{Im}(\X-\lambda\v\v^{\top})\subseteq\textrm{Im}(\X)$. Also,
\begin{align*}
\left({\X-\frac{1}{\v^{\top}\X^{\dagger}\v}\v\v^{\top}}\right)
\X^{\dagger}\v
=
\X\X^{\dagger}\v-\v
=
\mathbf{0}.
\end{align*}
The non-zero vector $\X^{\dagger}\v\in\textrm{Im}(\X)$ therefore enters the kernel after the update. A rank-one modification can reduce the rank by at most one, so the rank is reduced exactly by one.
\end{proof}

\section{Proof of Lemma \ref{lem:factor-comparison}}\label{app:procrustes}

\begin{proof}[Proof of Lemma \ref{lem:factor-comparison}]
Let $\bR$ be an orthogonal solution of the Procrustes problem in \eqref{eq:factor-comparison}, and define
$\H:=\Y-\Z\bR$.
Replacing $\Z$ with $\Z\bR$, we may assume without loss of generality that $\bR=\I$, that $\Y^{\top}\Z\succeq0$, and that $\H^{\top}\Z=\Z^{\top}\H$ is symmetric. Set
$\eta:=2(\sqrt{2}-1)\sigma_r(\Z)^2$. 
A direct expansion gives
\begin{align}
&\Vert{\Y\Y^{\top}-\Z\Z^{\top}}\Vert_F^2
-
\eta\Vert{\H}\Vert_F^2
\nonumber\\
&\qquad=
\trace\left({
(\H^{\top}\H)^2
+4\H^{\top}\H\H^{\top}\Z
+2(\H^{\top}\Z)^2
+2\Z^{\top}\Z\H^{\top}\H
-
\eta\H^{\top}\H
}\right).
\label{eq:procrustes-expansion}
\end{align}
The right-hand side of \eqref{eq:procrustes-expansion} can be written as
\begin{align*}
&\trace\left({
(\H^{\top}\H+\sqrt{2}\H^{\top}\Z)^2
}\right)
\\
&\quad+
\trace\left({
\H^{\top}\H
\left({
(4-2\sqrt{2})\H^{\top}\Z
+2\Z^{\top}\Z
-
\eta\I
}\right)
}\right).
\end{align*}
The first term is non-negative. For the second term, use $\H^{\top}\Z=\Y^{\top}\Z-\Z^{\top}\Z$ to obtain
\begin{align*}
&(4-2\sqrt{2})\H^{\top}\Z
+2\Z^{\top}\Z
-
\eta\I
\\
&\qquad=
(4-2\sqrt{2})\Y^{\top}\Z
+
2(\sqrt{2}-1)
\left({\Z^{\top}\Z-\sigma_r(\Z)^2\I}\right)
\succeq0.
\end{align*}
Here we used $\Y^{\top}\Z\succeq0$ and
$\Z^{\top}\Z\succeq\sigma_r(\Z)^2\I$. Thus,
\begin{align*}
\Vert{\Y\Y^{\top}-\Z\Z^{\top}}\Vert_F^2
\geq
2(\sqrt{2}-1)\sigma_r(\Z)^2\Vert{\H}\Vert_F^2,
\end{align*}
which is equivalent to \eqref{eq:factor-comparison}.
\end{proof}

\end{document}